\documentclass[10pt, a4paper, reqno]{amsart}
\usepackage{blindtext,amsmath,amsthm,amsfonts,amssymb}
\usepackage{mathtools}
\DeclarePairedDelimiter \floor {\lfloor} {\rfloor}
\usepackage{textcomp}
\usepackage{url}
\usepackage[utf8]{inputenc}
\usepackage{mathrsfs}
\usepackage{tikz}
\usetikzlibrary{positioning}
\input amssym.def
\usepackage[english]{babel}
\usepackage[autostyle, english = american]{csquotes}\MakeOuterQuote{"}
\newcommand\numberthis{\addtocounter{equation}{1}\tag{\theequation}}

\newtheorem{theorem}{Theorem}[section]

\theoremstyle{definition}

\newtheorem*{theorem*}{Theorem}
\newtheorem*{conjecture*}{Conjecture}
\newtheorem*{proposition*}{Proposition}
\newtheorem*{lemma*}{Lemma}

\newtheorem{corollary}[theorem]{Corollary}

\theoremstyle{definition}

\newtheorem*{remark*}{Remark}

 2

\begin{document}
\title{A VANISHING CRITERION FOR DIRICHLET SERIES WITH PERIODIC COEFFICIENTS}

\author[Tapas Chatterjee, M. Ram Murty and Siddhi Pathak]{Tapas Chatterjee\textsuperscript{1}, M. Ram Murty\textsuperscript{2} and Siddhi Pathak\textsuperscript{3}}
\email[Tapas Chatterjee]{tapas@iitrpr.ac.in}

\email[M. Ram Murty]{murty@mast.queensu.ca}

\email[Siddhi Pathak] {siddhi@mast.queensu.ca}



\footnotetext[1]{Research of the first author was supported by ISIRD grant at the Indian Institute of Technology Ropar.}

\footnotetext[2]{Research of the second author was supported by an NSERC Discovery grant.}
\maketitle

\begin{abstract}
We address the question of non-vanishing of $L(1,f)$ where $f$ is an algebraic-valued, periodic arithmetical function. We do this by characterizing algebraic-valued, periodic functions $f$ for which $L(1,f)=0$. The case of odd functions was resolved by Baker, Birch and Wirsing in 1973. We apply a result of Bass to obtain a characterization for the even functions. We also describe a theorem of the first two authors which says that it is enough to consider only the even and the odd functions in order to obtain a complete characterization.
\end{abstract}

\section{\bf Introduction}
\bigskip

To unravel the mysteries surrounding Dirichlet's theorem about non-vanishing of $L(1,\chi)$ for a non-principal Dirichlet character $\chi$, Sarvadaman Chowla \cite{chowla} made the following conjecture in the early 1960s:
\begin{conjecture*}
Let $f$ be a rational-valued arithmetical function, periodic with prime period $p$. Further assume that $f(p) = 0$ and
\begin{equation*}
\sum_{a=1}^p f(a) = 0.
\end{equation*}
Then,
\begin{equation*}
\sum_{n=1}^{\infty} \frac{f(n)}{n} \neq 0,
\end{equation*}
unless $f$ is identically zero.
\end{conjecture*}
Chowla proved this conjecture in the case of odd functions i.e, $f(p-n) = -f(n)$ based on an outline of the proof by Siegel \cite{chowla}. The complete resolution of Chowla's question in a wider setting was given by Baker, Birch and Wirsing in 1973 \cite{bbw}. They proved the following general theorem:
\begin{theorem*}
If $f$ is a non-vanishing function defined on the integers with algebraic values and period $q$ such that (i) $f(n) = 0$ whenever $1 < (n,q) < q$ and (ii) the $q^{th}$ cyclotomic polynomial $\Phi_q$ is irreducible over $\mathbb{Q} (f(1), f(2), \cdots , f(q))$, then
\begin{equation*}
\sum_{n=1}^\infty \frac{f(n)} {n} \neq 0.
\end{equation*} 
\end{theorem*}
Let us observe that in the case of Chowla's conjecture, condition (i) is vacuous as $q$ is prime and $f(q)=0$. The condition (ii) is also satisfied as $f$ is rational-valued and the $q$-th cyclotomic polynomial is irreducible over $\mathbb{Q}$ . Thus, the Baker-Birch-Wirsing theorem implies Chowla's conjecture.

Chowla's question can be asked in the most general setting as follows: Fix a positive integer $q$. Let $f$ be an algebraic-valued arithmetical function, periodic with period $q$. It is useful to define an $L$-function associated to the function $f$, namely,
\begin{equation*}
L(s,f) = \sum_{n=1}^{\infty} \frac{f(n)}{n^s}.
\end{equation*}
Using the analytic continuation of the Hurwitz zeta function \cite{hurwitz}, we can deduce that $L(1,f)$ exists if and only if $\sum_{a=1}^q f(a) = 0$. Thus, we can ask the following question: if $f$ is not identically zero, then is it true that
\begin{equation*}
\sum_{n=1}^{\infty} \frac{f(n)}{n} \neq 0 \hspace{1mm}?
\end{equation*}
The answer in this case turns out to be negative \cite{bbw}. As an example, consider the function $f$ defined such that
\begin{equation}\label{function-eg}
\sum_{n=1}^{\infty} \frac{f(n)}{n^s} = {(1 - p^{(1-s)})}^2 \zeta(s).
\end{equation}
In particular,
\begin{equation*}
f(n) = 
\begin{cases}
1 & \text{if } (n,p)=1, \\
1 - 2p & \text{if } p | n, p^2 \nmid n, \\
{(p-1)}^2 & \text{otherwise}.
\end{cases}
\end{equation*}
Note that the function $f$ is periodic with period $p^2$. Taking limit of the right hand side of \eqref{function-eg} as $s \rightarrow 1$, we get
\begin{equation*}
L(1,f) = 0,
\end{equation*}
because $\zeta(s)$ has a simple pole at $s=1$.

In their paper, \cite{bbw}, Baker, Birch and Wirsing also give a characterization of all odd algebraic-valued periodic arithmetical functions $f$ that satisfy $L(1,f) = 0$. Since their argument is short and elegant, we describe it in the third section. Their approach suggests a change in perspective. Instead of trying to prove the non-vanishing of an expression, we will try to characterize the functions $f$ for which $L(1,f)=0$.

Let $f$ be any function. Then it can be written as the sum of an even function and an odd function as follows. Define $f_o(a) := [f(a) - f(-a)]/2$  and $f_e(a) := [f(a) + f(-a)]/2$. Then clearly $f = f_o + f_e$. In 2014, Ram Murty and Tapas Chatterjee \cite{ram-tapas} made an important observation. They proved,
\begin{theorem}\label{even-odd}
For a periodic, algebraic-valued arithmetical function $f$,
\begin{equation*}
L(1,f) = 0 \iff L(1,f_o) = 0 \hspace{2mm} \text{and} \hspace{2mm} L(1,f_e) = 0.
\end{equation*}
\end{theorem}
In the view of this theorem, to understand the vanishing of $L(1,f)$ in our case, it is enough to consider even algebraic-valued periodic arithmetical functions. In the fourth section, we apply a beautiful result of Bass \cite{bass} to obtain a set of functions that act as building blocks for even algebraic-valued periodic arithmetical functions $f$ with $L(1,f) = 0$. This completes to an extent the characterization we were after.

\section{\bf Preliminaries}
\bigskip

The aim of this section is to introduce notation and some fundamental results that will be used in the later part of the paper.

Let $q$ be a fixed positive integer. Consider $f : \mathbb{Z} \to \bar{\mathbb{Q}}$, periodic with period $q$. Define
\begin{equation*}
L(s,f) = \sum_{n=1}^{\infty} \frac{f(n)}{n^s}.
\end{equation*}
Let us observe that $L(s,f)$ converges absolutely for $\Re(s) > 1$. Since $f$ is periodic,
\begin{equation*}\label{hurw-form}
\begin{split}
L(s,f) & = \sum_{a=1}^q f(a)\sum_{k=0}^{\infty} \frac{1}{{(a + kq)}^s}\\
& = \frac{1}{q^s} \sum_{a=1}^q f(a) \zeta(s, a/q),
\end{split}
\end{equation*}
where $\zeta(s,x)$ is the Hurwitz zeta function. For $\Re(s) > 1$ and $0 < x \leq 1$, the Hurwitz zeta function is defined as
\begin{equation*}
\zeta(s,x) = \sum_{n=0}^{\infty} \frac{1}{{(n+x)}^s}.
\end{equation*}
In 1882, Hurwitz \cite{hurwitz} proved that $\zeta(s,x)$ has an analytic continuation to the entire complex plane except for a simple zero at $s=1$ with residue $1$. This can be used to conclude that $L(s,f)$ can be extended analytically to the entire complex plane except for a simple pole at $ s=1 $ with residue $ \frac {1} {q} \sum_{a=1}^q f(a)$. Thus, $\sum_{n=1}^\infty \frac {f(n)} {n}$ exists whenever  $\sum_{a=1}^q f(a) = 0$, which we will assume henceforth. Thus, $L(s,f)$ is an entire function.

Given a function $f$ which is periodic mod $q$, we define the Fourier transform of $f$ as
\begin{equation*}
\hat{f}(x) := \frac{1}{q} \sum_{a=1}^q f(a) \zeta_q^{-ax}, 
\end{equation*}
where $\zeta_q = e^{2 \pi i/q}$.
This can be inverted using the identity
\begin{equation}\label{fourier-inversion}
f(n) = \sum_{x=1}^q \hat{f}(x) \zeta_q^{xn}.
\end{equation}
Thus, the condition for convergence of $L(1,f)$, i.e, $\sum_{a=1}^q f(a) = 0$ can be interpreted as $\hat{f}(q) = 0$. Substituting \eqref{fourier-inversion} in the expression for $L(s,f)$ we have,
\begin{equation*}
\begin{split}
L(s,f) & = \sum_{n=1}^{\infty} \frac{1}{n^s} \sum_{x=1}^q \hat{f}(x) \zeta_q^{xn}.\\
 & = \sum_{x=1}^{q} \hat{f}(x) \sum_{n=1}^{\infty} \frac{\zeta_q^{xn}}{n^s}.
\end{split}
\end{equation*}
Assuming that $\hat{f}(q) = 0$, specializing at $s = 1$ and using the Taylor series expansion of the log we conclude that
\begin{equation}\label{linear-form-log}
L(1,f) = - \sum_{x=1}^{q-1} \hat{f}(x) \log(1 - \zeta_q^x),
\end{equation}
where $\log$ is the principal branch.

We will also apply the famous theorem of Baker \cite{baker} concerning linear forms in logarithms and so we note it here.
\begin{theorem}\label{baker-llog}
If $\alpha_1, \alpha_2, \cdots, \alpha_n$ are non-zero algebraic numbers, such that $\log \alpha_1, \log \alpha_2, \cdots , \log \alpha_n$ are $\mathbb{Q}$-linearly independent, then $1 , \log \alpha_1, \cdots, \log \alpha_n$ are $\bar{\mathbb{Q}}$-linearly independent.
\end{theorem}
A useful corollary of this statement is the following.
\begin{corollary}\label{baker-cor}
If $\alpha_1, \alpha_2, \cdots , \alpha_n$ are algebraic numbers different from $0$ and $1$, and $\beta_1,\beta_2, \cdots , \beta_n$ are $\mathbb{Q}$-linearly independent algebraic numbers then,
\begin{equation*}
\beta_1 \log{\alpha_1} + \beta_2 \log{\alpha_2} + \cdots + \beta_n \log{\alpha_n}\\
\end{equation*}
is non-zero and hence, transcendental.
\end{corollary}
Another helpful corollary of Baker's theorem is the following statement about linear forms in logarithms of positive algebraic numbers. For a proof of the corollary, we refer the reader to \cite{ram-saradha}.
\begin{corollary}\label{transc-pi}
Let $\alpha_1, \cdots, \alpha_n$ be positive algebraic numbers. If $c_0, \cdots, c_n$ are algebraic numbers and $c_0 \neq 0$, then
\begin{equation*}
c_0 \pi + \sum_{j=1}^n c_j \log \alpha_j
\end{equation*}
is non-zero and hence, a transcendental number.
\end{corollary}
\begin{remark*}
The proof for the above corollary given in \cite{ram-saradha} also goes through when the branch of logarithm chosen is the principal branch. In that case, we replace $i \pi$ by $2 \log i$ and proceed as in \cite{ram-saradha}.
\end{remark*}

In an unpublished paper, Milnor conjectured the complete set of multiplicative relations among cyclotomic numbers in the set $\{ 1 - \zeta_q^x : 1 \leq x \leq q-1 \}$ for a fixed positive integer $q$. Here $\zeta_q$ denotes a primitive $q^{\text{th}}$ root of unity. This conjecture was proved by Hyman Bass \cite{bass} in 1965. In 1972, Veikko Ennola \cite{ennola} realized that the conjecture was true upto a factor of $2$ and not in general. He gave a different proof of the conjecture in his paper. Since his formulation of the theorem is easier to apply in our setting, we will state it here. Let $a_x := \log(|1 - \zeta_q^x|)$. The following theorem characterizes all additive relations among the numbers $\{a_x | 1 \leq x \leq q-1 \}$.
\begin{theorem}\label{bass-ennola}
Consider the following two relations: For $1 \leq x \leq \floor*{\frac{(q-1)}{2}}$,
\begin{equation}\label{relation1}
\mathfrak{R}_1 : a_x - a_{q-x} = 0,
\end{equation}
and for any divisor $d$ of $q$ and $1 < d < q$ and $1 \leq c \leq d-1$,
\begin{equation}\label{relation2}
\mathfrak{R}_2: a_{\frac{q}{d}c} - \sum_{j=1}^{\frac{q}{d} - 1} a_{c+dj} = 0.
\end{equation}
Let $R$ be an additive relation among the $a_x$'s over the integers. Then, $2R$ is a $\mathbb{Z}$-linear combination of relations of the form $\mathfrak{R}_1$ and $\mathfrak{R}_2$.
\end{theorem} 

\section{\bf Odd functions}
\bigskip

In this section, we reproduce a simple argument of Baker, Birch and Wirsing \cite{bbw} that gives us a necessary and sufficient condition on an odd function $f$ such that $L(1,f) = 0$.
\begin{theorem}\label{odd-function-char}
Let $f$ be an odd algebraic-valued arithmetical function, periodic with period $q$. Then $L(1,f) = 0$ if and only if $\sum_{x=1}^{q-1} x \hat{f}(x) = 0$.
\end{theorem}
\begin{proof}
Let us note that
\begin{equation}\label{mod-of-cycl}
1 - \zeta_q^x = - (\zeta_q^{x/2} - \zeta_q^{-x/2}) \zeta_q^{x/2} = -2 i \bigg( \sin \bigg(\frac{x \pi}{q} \bigg)\bigg)e^{x \pi i/q},
\end{equation}
and so the principal value of the logarithm is
\begin{equation}\label{log-real-and-im}
\log (1 - \zeta_q^x) = \log \bigg(2 \sin \frac{x \pi}{q} \bigg) + \bigg(\frac{x}{q} - \frac{1}{2} \bigg) \pi i
\end{equation}
for $1 \leq x < q$. Substituting \eqref{log-real-and-im} in the expression for $L(1,f)$ as a linear form in logarithms \eqref{linear-form-log}, we get
\begin{align*}
L(1,f) & = - \sum_{x=1}^{q-1}  \hat{f}(x) \bigg[ \log \bigg(2 \sin \frac{x \pi}{q} \bigg) + \bigg(\frac{x}{q} - \frac{1}{2} \bigg) \pi i \bigg] \\
& =  - \sum_{x=1}^{q-1} \hat{f}(x) \log \bigg(2 \sin \frac{x \pi}{q} \bigg) -  \frac{i \pi}{q}\sum_{x=1}^{q-1} x \hat{f}(x)  + \frac{i \pi}{2} \sum_{x=1}^{q-1} \hat{f}(x). \numberthis \label{cond-main} 
\end{align*}
Since $f$ is an odd function, $\hat{f}$ is also an odd function. Hence,
\begin{equation*}
2 \sum_{x=1}^{q-1} \hat{f}(x) = \sum_{x=1}^{q-1} [ \hat{f}(x) + \hat{f} (q-x) ] = 0.
\end{equation*}
Therefore, the last term of \eqref{cond-main} is zero. Now, note that $\sin( \pi - \theta) = \sin( \theta)$. Thus, $\sin(x \pi/q)$ is an even function. Hence, $\log (2 \sin \frac{x \pi}{q} )$ is even which implies that $\hat{f}(x) \log (2 \sin \frac{x \pi}{q} )$ is an odd function.\\
Therefore, the first term of \eqref{cond-main},
\begin{equation*}
\sum_{x=1}^{q-1} \hat{f}(x) \log \bigg( 2 \sin \frac{x \pi}{q} \bigg)= 0. 
\end{equation*}
The result is immediate from here. 
\end{proof}
\begin{remark*}
The condition obtained above is on the Fourier transform of the function and not the function itself. Using the Fourier inversion formula, we can deduce a condition on the function. The condition obtained in Theorem \ref{odd-function-char} can be simplified to
\begin{equation*}
\begin{split}
\sum_{x=1}^{q-1} x \hat{f}(x) = & \frac{1}{q}\sum_{x=1}^{q-1} x \sum_{n=1}^q f(n) \zeta_q^{-nx} \\
& = \frac{1}{q} \sum_{n=1}^q f(n) \sum_{x=1}^{q-1} x \zeta_q^{-nx} \\
& = \frac{1}{q} \sum_{n=1}^{q-1} f(n) \sum_{x=1}^{q-1} x \zeta_q^{-nx},
\end{split}
\end{equation*}
as $f(q) = f(0) = f(-q) = - f(q) = 0$ since $f$ is odd. The innermost sum can be evaluated as follows. Let $T$ be an indeterminate. Observe that 
\begin{equation}\label{eq1}
\sum_{x=0}^{q-1} T^x = \frac{T^q - 1}{T - 1}.
\end{equation}
Differentiating \eqref{eq1} with respect to $T$, we have
\begin{equation*}
\sum_{x=1}^{q-1} x T^{x-1} = \frac{qT^{q-1}}{{T-1}} - \frac{T^q - 1}{{(T-1)}^2}.
\end{equation*}
Multiplying the above equation by $T$ and substituting $T = \zeta_q^{-n}$, we get
\begin{equation*}
\sum_{x=1}^{q-1} x \zeta_q^{-nx} = \frac{q}{\zeta_q^{-n} - 1} = \frac{q \zeta_q^n}{1 - \zeta_q^n}.
\end{equation*}
This observation along with Theorem \ref{odd-function-char} gives: for $f$ odd,
\begin{equation}\label{condition}
L(1,f) = 0 \iff \sum_{n=1}^{q-1} \frac{f(n)}{1 - \zeta_q^n} = 0.
\end{equation}
Let us note that
\begin{equation*}
\frac{1}{1 - \zeta_q^n} = \frac{i}{2} \cot \bigg( \frac{n \pi}{q} \bigg) + \frac{1}{2}
\end{equation*}
and that since $f$ is odd,
\begin{equation*}
2 \sum_{n=1}^{q-1} f(n) = \sum_{n=1}^{q-1} [ f(n) + f(q-n)] = 0.
\end{equation*}
Hence, the condition \eqref{condition} can be translated as: For odd algebraic-valued periodic functions $f$,
\begin{equation*}
L(1,f) = 0 \iff \sum_{n=1}^{q-1} f(n) \cot \bigg( \frac{n \pi}{q} \bigg) = 0.
\end{equation*}
\end{remark*}
We would like to mention that Baker, Birch and Wirsing also obtained a basis for the $\bar{\mathbb{Q}}$-vector space of odd algebraic-valued arithmetical functions $f$, periodic with period $q$ and $L(1,f) = 0$ in their paper \cite{bbw}. Thus, the characterization of odd functions is complete.

\section{\bf Even functions}
\bigskip

In this section, we will use a theorem of Bass that characterizes all the multiplicative relations among cyclotomic numbers modulo torsion (\cite{bass}, \cite{ennola}). We will give a necessary condition for even periodic functions $f$ to satisfy $L(1,f) = 0$.\\\\
We define the following functions that serve as building blocks for even algebraic-valued functions, periodic with period $q$.
Fix a divisor $d$ of $q$. For $c \in \{ 1, 2, \cdots, d-1 \}$, we define $F_{d,c} := F^{(1)}_{d,c} - F^{(2)}_{d,c}$ where the two functions, $F^{(1)}_{d,c}$ and $F^{(2)}_{d,c}$ are arithmetical functions, periodic with period $q$. They are defined as follows:
\begin{equation*}
F^{(1)}_{d,c} (x) = 
\begin{cases}
1/2 & \text{if } x \equiv c \bmod q,\\
0 & \text{otherwise}.
\end{cases}
\end{equation*}

\begin{equation*}
F^{(2)}_{d,c} (x) =
\begin{cases}
1/2 & \text{if } x \equiv (\frac {q} {d}) c \bmod q,\\
0 & \text{otherwise}.
\end{cases}
\end{equation*}

We will prove the following result in this section:
\begin{theorem}\label{even-function-char}
Let $f$ be an algebraic valued, even function which is periodic with period $q$. If $L(1,f) = 0$, then $f$ is an algebraic linear combination of the functions $\{\widehat{F_{d,c}} | \textrm{ for any divisor} \ d \ \textrm{of} \  q, \hspace{1mm} 1 < d < q, \hspace{1mm} 1 \leq c \leq d - 1 \}$.
\end{theorem}
Here, $\widehat{F_{d,c}}$ denotes the Fourier transform of $F_{d,c}$ which can be computed as follows:
For $1 \leq y \leq q$,
\begin{equation*}
\begin{split}
\widehat{F_{d,c}^{(1)}}(y) & =  \frac{1}{q} \sum_{a=1}^q F_{d,c}^{(1)}(a) \zeta_q^{-ay} \\
& = \frac{1}{2q} \sum_{j=0}^{\frac{q}{d} - 1} \zeta_q^{-(c + dj)y} \\
& = \frac{\zeta_q^{-yc}}{2q} \sum_{j=0}^{\frac{q}{d} -1} \zeta_q^{-djy} \\
& = \frac{\zeta_q^{-yc}}{2q} \sum_{j=0}^{\frac{q}{d} -1} \zeta_{\frac{q}{d}}^{-jy}.
\end{split}
\end{equation*}
Note that the sum
\begin{equation*}
\sum_{j=0}^{\frac{q}{d} -1} \zeta_{\frac{q}{d}}^{-jy} =
\begin{cases}
\frac{q}{d} & \text{if } y \equiv 0 \bmod \frac{q}{d}, \\
0 & \text{otherwise}.
\end{cases}
\end{equation*}
Similarly, for $1 \leq y \leq q$,
\begin{equation*}
\begin{split}
\widehat{F_{d,c}^{(2)}}(y) & = \frac{1}{q} \sum_{a=1}^q F_{d,c}^{(2)} (a) \zeta_q^{-ay} \\
& = \frac{1}{2q} \zeta_q^{-\frac{q}{d} c y} \\
& = \frac{1}{2q} \zeta_d^{-cy}.
\end{split}
\end{equation*}

More precisely,
\begin{equation*}
\widehat{F_{d,c}} (y) =
\begin{cases}
\frac{\zeta_q^{-cy}}{2d} - \frac{\zeta_d^{-cy}}{2q} & \text{if } y \equiv 0 \bmod \frac{q}{d}, \\
- \frac{1}{2q} \zeta_d^{-cy} & \text{otherwise}.
\end{cases}
\end{equation*}
Thus, we note that $\widehat{F_{d,c}}$ 's are in fact, simple functions.
\begin{proof}
Let $f$ be an even algebraic-valued, periodic function with period $q$, not identically zero. Let $M := \mathbb{Q}(f(1), \cdots, f(q), \zeta_q)$. Let $\left \{ \omega_1, \omega_2, \cdots, \omega_r \right \}$  be a basis for $M$ over $\mathbb{Q}$. There exists  $d_j(x)$  $\in$  $\mathbb{Q}$ such that,
\begin{equation*}
\hat{f}(x) = \sum_{j=1}^r d_j(x) \omega_j. 
\end{equation*}
We can choose an integer $N$ such that $\forall$ $1 \leq x \leq q$, $\forall$ $1 \leq j \leq r$, $c_j(x) := Nd_j(x) \in \mathbb{Z}.$
Hence, $N\hat{f}(x) = \sum_{j=1}^r c_j(x) \omega_j$. Thus,
\begin{equation*}
\begin{split}
NL(1,f) & = - \sum_{x=1}^{q-1} \sum_{j=1}^r \omega_j c_j(x) \log { (1 - \zeta_q^x) } \\
& = - \sum_{j=1}^r \omega_j \sum_{x=1}^{q-1} c_j(x) \log { (1 - \zeta_q^x) }
\end{split}
\end{equation*}
Let
\begin{equation*}
R_j := \sum_{x=1}^{q-1} c_j(x) \log { (1 - \zeta_q^x) }.
\end{equation*}
Therefore,
\begin{equation*}
-NL(1,f) = \sum_{j=1}^r \omega_j R_j.
\end{equation*}
As $\hat{f}$ is even and $\omega_1, \cdots, \omega_r$ is a basis, 
\begin{equation*}
c_j(x) = c_j(q-x).
\end{equation*}
Therefore,
\begin{equation*}
R_j = \sum_{x=1}^{\floor*{\frac{(q-1)}{2}}} c_j(x) [ \log (1 - \zeta_q^x) + \log (1 - \zeta_q^{-x}) ].
\end{equation*}
Note that $\log$ denotes the principal branch of logarithm. Thus, $\arg (1 - \zeta_q^x) = - \arg (1 - \zeta_q^{-x})$. For $ 1 \leq x \leq q-1$, define $a_x := \log (|1 - \zeta_q^x|)$. Thus,
\begin{equation*}
R_j = \sum_{x=1}^{\floor*{\frac{(q-1)}{2}}} 2 c_j(x) a_x.
\end{equation*}
Since $c_j(x) = c_j(q-x)$, $R_j$ can be written as
\begin{equation*}
R_j = \sum_{x=1}^{q-1} c_j(x) a_x = \log \bigg( \prod_{x=1}^{q-1} {(|1 - \zeta_q^x|)}^{c_j(x)} \bigg).
\end{equation*}
Let
\begin{equation*}
\alpha_j := \prod_{x=1}^{q-1} {(|1 - \zeta_q^x|)}^{c_j(x)}.
\end{equation*}
Therefore, \\
\begin{equation}\label{L(1,f)-linear-log}
(-N) L(1,f) = \sum_{j=1}^r \omega_j \log \alpha_j.
\end{equation}
Let us note that $\alpha_j$ is non-zero, algebraic. Thus, by Corollary \ref{baker-cor}, if $\alpha_j \neq 1$ for some $1 \leq j \leq r$, then $L(1,f)$ will be transcendental and hence non-zero. But $L(1,f) = 0$ by assumption. Hence, $\alpha_j = 1$ and in turn $R_j = 0$ ,  $\forall$ $1 \leq j \leq r$. Thus, we are led to consider relations among logarithms of the cyclotomic numbers, $1 - \zeta_q^x$. Let $\mathscr{R}$ denote the relation $2R_j = 0$. By Theorem \ref{bass-ennola},  $\mathscr{R}$ belongs to the $\mathbb{Z}$-module generated by relations of the form \eqref{relation1} and \eqref{relation2}. Since $c_j(x) = c_j(q-x)$, $\mathscr{R}$ belongs to the $\mathbb{Z}$-module generated by \eqref{relation2}. Indeed all relations $R:= \sum_{x=1}^{q-1} C_x a_x = 0$ in the $\mathbb{Z}$-module generated by \eqref{relation1} satisfy $C_x = - C_{q-x}$, which along with the fact that $C_j(x) = C_j(q-x)$ (which stem from the evenness of $f$) imply that $c_j(x)=0$ $\forall$ $1 \leq x \leq q-1$. Thus, $\mathscr{R}$ is a $\mathbb{Z}$-linear combination of \eqref{relation2}. Observe that the functions $F_{d,c}$ are precisely those that represent the relation \eqref{relation2}. This implies that the functions $c_j$ are integer linear combinations of $F_{d,c}$, say $c_j(x) = \sum_{d|q, 1 < d < q} \sum_{c=1}^{d-1} m_{j,d,c} F_{d,c}(x)$, where $m_{j,d,c} \in \mathbb{Z}$. Then,
\begin{equation*}
\begin{split}
N \hat{f}(x)& = \sum_{j=1}^{r} c_j(x) \omega_j.\\
 & = \sum_{j=1}^{r} \omega_j \sum_{d|q \atop 1 < d < q} \sum_{c=1}^{d-1} m_{j,d,c} F_{d,c}(x).\\
 & = \sum_{d|q \atop 1 < d < q} \sum_{c=1}^{d-1} F_{d,c}(x) \sum_{j=1}^{r} m_{j,d,c} \omega_j.\\
\end{split}
\end{equation*}
Taking the Fourier transform of both sides and using the Fourier inversion formula, the result follows.
\end{proof}
\begin{remark*}
On examining the above proof, we observe that we also obtain a sufficiency condition for $L(1,f)=0$. More precisely, suppose $f$ is a given algebraic linear combination of the functions
\begin{equation*}
\{\widehat{F_{d,c}} \hspace{1mm} |\  \textrm{any divisor} \ d \ \textrm{of} \  q, 1 < d < q, 1 \leq c \leq d - 1 \} ,
\end{equation*}
such that $f$ is an even function, then $L(1,f) = 0$. But it does not seem easy to determine whether a given function is an algebraic linear combination of these particular functions.
\end{remark*}

As a corollary, we have:
\begin{corollary}
When $q$ is prime, there are no algebraic-valued even functions $f$ such that $L(1,f)=0$.
\end{corollary}
\begin{proof}
If $q$ is prime, the set $\{\widehat{F_{d,c}} |\  \textrm{any divisor} \ d \ \textrm{of} \  q, 1 < d < q, 1 \leq c \leq d - 1 \}$ is an empty set. The conclusion follows from Theorem \ref{even-function-char}.
\end{proof}

\section{\bf Conclusion}
\bigskip

The previous two sections give a characterization of algebraic-valued periodic arithmetic functions, odd and even respectively whose $L$-series vanish at $s=1$. In this section, we mention a theorem of Ram Murty and Tapas Chatterjee \cite{ram-tapas} which nicely ties the non-vanishing of the odd and the even functions to give us the characterization required. The proof in their paper is incorrect as care was not taken regarding the branch of logarithm. The proof given below follows a different argument. As seen earlier, given any function $f$, we write it as $f = f_o + f_e$, where $f_o$ is the odd part of $f$ and $f_e$ is the even part of $f$.
\begin{theorem}
If $f$ is an algebraic-valued arithmetical function, periodic with period $q$, then
$L(1,f) = 0$ $\iff$ $L(1,f_o) = 0$ and $L(1,f_e) = 0$.
\end{theorem}
\begin{proof}
Let us note that
\begin{equation*}
\sum_{n=1}^{\infty} \frac{f(n)}{n} = \sum_{n=1}^{\infty} \frac{f_e(n)}{n} + \sum_{n=1}^{\infty} \frac{f_o(n)}{n}.
\end{equation*}
Thus, it is clear that if $L(1,f_o)$ and $L(1,f_e)$ are both zero, then so is $L(1,f)$. 

Suppose now that $L(1,f) = 0$. The arguments in the second section imply
\begin{equation}\label{L(1,f-odd)}
L(1,f_o) = \frac{-i \pi}{q} \sum_{x=1}^{q-1} x \hat{f_o} (x).
\end{equation} 
We claim that, 
\begin{equation}\label{L(1,f-even)}
L(1,f_e) = -2 \sum_{x=1}^{\floor*{\frac{(q-1)}{2}}} \hat{f_e}(x) \log \bigg( 2  \sin \frac{x \pi}{q}  \bigg).
\end{equation}
Indeed, by \eqref{linear-form-log}, we know that
\begin{equation*}
\sum_{n=1}^{\infty} \frac{f_e(n)}{n} = - \sum_{x=1}^{q-1} \hat{f_e}(x) \log ( 1 - \zeta_q^x).
\end{equation*}
Since $f_e$ is even, so is $\hat{f_e}$, i.e, $\hat{f_e}(q-x) = \hat{f_e} (x)$. Thus,
\begin{equation*}
L(1,f_e) = -2 \sum_{x=1}^{\floor*{\frac{(q-1)}{2}}} \hat{f_e}(x) \log  | 1 - \zeta_q^x|.
\end{equation*}
By \eqref{mod-of-cycl}, 
\begin{equation*}
| 1 - \zeta_q^x | = \bigg| 2 \sin \bigg( \frac{x \pi}{q} \bigg) \bigg|.
\end{equation*}
Hence,
\begin{equation*}
L(1,f_e) = -2 \sum_{x=1}^{\floor*{\frac{(q-1)}{2}}} \hat{f_e}(x) \log \bigg| 2 \sin \bigg( \frac{x \pi}{q} \bigg) \bigg| .
\end{equation*}
Note that, for $1 \leq x \leq \floor*{\frac{(q-1)}{2}}$, $\sin( x \pi/q) > 0$. 

Thus \eqref{L(1,f-odd)} and \eqref{L(1,f-even)} imply
\begin{equation}\label{final-log}
L(1,f) = \frac{-i}{q} \bigg( \sum_{x=1}^{q-1} x \hat{f_o} (x) \bigg) \pi -2 \sum_{x=1}^{\floor*{\frac{(q-1)}{2}}} \hat{f_e}(x) \log \bigg( 2  \sin \frac{x \pi}{q}  \bigg).
\end{equation}

Suppose $L(1,f_o) \neq 0$. By Corollary \ref{transc-pi}, $L(1,f) \neq 0$, which contradicts our assumption. Thus, $L(1,f_o) = 0$ and hence, $L(1,f_e) = 0$. 
\end{proof}
This completes the characterization of algebraic-valued, periodic arithmetical functions with $L(1,f) = 0$.

\end{document}